\documentclass[letterpaper, 10 pt, conference]{ieeeconf}  

\usepackage{url}
\usepackage{amsmath}%
\usepackage{amsfonts}%
\usepackage{amssymb}%
\usepackage{bm}%
\usepackage{graphicx}
\usepackage{caption}
\usepackage{subfigure}
\usepackage{cite}
\usepackage{color}
\usepackage{stmaryrd}
\usepackage[ruled,vlined,noend]{algorithm2e}
\usepackage{hhline}
\newtheorem{theorem}{Theorem}

\newcounter{lemma_}
\newtheorem{lemma}[lemma_]{Lemma}
\newcounter{proposition_}
\newtheorem{proposition}[proposition_]{Proposition}
\newcounter{corollary_}
\newtheorem{corollary}[corollary_]{Corollary}
\newcounter{remark_}
\newtheorem{remark}[remark_]{Remark}

\newenvironment{proof2}[1][Proof]{\textbf{#1.} }{\ \rule{0.5em}{0.5em}}

\makeatother
\IEEEoverridecommandlockouts                             

\makeatletter
\newcommand{\fixed@sra}{$\vrule height 2\fontdimen22\textfont2 width 0pt\shortrightarrow$}
\newcommand{\shortarrow}[1]{%
	\mathrel{\text{\rotatebox[origin=c]{\numexpr#1*45}{\fixed@sra}}}
}

\title{\Large \bf Optimal sampling schedules for $h_2$ and $h_\infty$ state-feedback control}

\author{Duarte J. Antunes and Jo\~ao P. Hespanha
\thanks{Duarte J. Antunes is with the Control Systems Technology Group, Dep. of Mechanical Eng., Eindhoven University of Technology, the Netherlands. {\tt\small d. antunes@tue.nl}. Jo\~ao P. Hespanha is with the Dep. of Electrical and Computer Eng., University of California, Santa Barbara, CA 93106-9560, USA. {\tt\small hespanha@ece.ucsb.edu}.  }} 

\begin{document}

\maketitle
\thispagestyle{empty}
\pagestyle{empty}

\begin{abstract}
	We consider a discrete-time linear system for which the control input is updated at every sampling time, but the state is measured at a slower rate. We allow the state to be sampled according to a periodic schedule, which dictates when the state should be sampled along a period.
Given a desired average sampling interval, our goal is to determine sampling schedules that are optimal in the sense that they minimize the $h_2$ or the $h_\infty$ closed-loop norm, under an optimal state-feedback control law. Our results show that, when the desired average sampling interval is an integer, the optimal state sampling turns out to be evenly spaced. This result indicates that, for the $h_2$ and $h_\infty$ performance metrics, there is relatively little benefit to go beyond constant-period sampling.
\end{abstract}
\section{Introduction}
\par The standard paradigm in digital control is to periodically sample the system's output, compute the control action, and update the system's input. Digital to analog and analog to digital converters often dictate these operations to occur at evenly spaced times, even if, occasionally, at different rates~\cite{cuenca:18}.  However, when, for example, the controller and sensor processing units ran on shared processors or control signals are transmitted over shared networks, different sensor and control update schedules are imposed or can be selected~\cite{cervin1}. Here we address how to optimize the sampling schedules.
\par We consider a discrete-time linear system for which the control input is updated at every time step and the state is sampled according to an arbitrary periodic schedule; each schedule is characterized by the intervals between consecutive samples in a period $h$. The cost of a sampling schedule is measured by the $h_2$ or $h_\infty$ closed-loop system norms under an optimal control law. We tackle the problem of picking optimal sampling schedules with a desired rational average sampling interval. 
 \par In the $h_2$ framework, we start by establishing a key result stating that the expected value of a quadratic cost in the interval between two samples is a convex function  of the length of the interval (in a natural sense for functions with discrete domains). Moreover, we  show that the $h_2$ norm can be written as a weighted average of samples of this convex function at the lengths of the intervals characterizing the periodic schedule. These two facts lead to a simple way to modify a schedule in order to decrease the associated $h_2$ norm: take two arbitrary intervals characterizing the schedule and reduce the largest by the same amount that the smallest is increased. This implies that:
 \begin{itemize}
 	\item[1)] An $h$-periodic schedule with $m$ intervals between sampling is optimal if (and only if under mild conditions) all of these intervals are either equal to $\lfloor \frac{h}{m}\rfloor$ or  $\lceil \frac{h}{m}\rceil$ where $\lfloor \cdot \rfloor$, $\lceil \cdot \rceil$ denote the floor and the ceil.
 \item[2)] When $\frac{h}{m}$  is an integer, then evenly sampled sampling is optimal and, in fact, it is the unique optimal schedule (under mild assumptions).
 \item[3)] The plot optimal achievable $h_2$ norm versus  average rate is a continuous piece-wise affine function connecting the pairs $(1/h,J_{2,h})$ where $J_{2,h}$ is the $h_2$ norm of periodic control with integer average sampling interval $h$ (see Figure~\ref{fig:3} below).
 \end{itemize}   The results are briefly connected to  existing results in the literature for continuous-time sampled data systems (see Remarks~\ref{rem:1},~\ref{rem:2} below). 
 \par In the $h_\infty$ framework,  we show that the $h_\infty$ norm only depends on the longest interval in a schedule, and it is a non-decreasing function of this longest interval. This implies that:
 \begin{itemize}
 \item[1)] A sampling schedule guarantees the smallest attenuation bound ($h_\infty$ norm) achievable for a given rational average sampling time $\frac{h}{m}$ if (and only if under mild assumptions) the largest interval does not exceed $\lceil \frac{h}{m}\rceil $. Note that the optimal schedules (in this sense) are in general different from the ones for the $h_2$ case. 
 \item[2)] Also here, when $\frac{h}{m}$ is an integer, then evenly sampled sampling is optimal and it is the unique optimal schedule (under mild assumptions).
 \item[3)] The plot optimal achievable $h_\infty$ norm versus average sampling time (or average rate) is a discontinuous piece-wise constant function, where these constants are equal to the $h_\infty$ norms corresponding to evenly spaced sampling with integer average sampling interval $h$ (see Figure~\ref{fig:4} below).
 \end{itemize}
 \par A numerical example illustrates the results.
\par There are some related results in the literature, reviewed next. The co-design of the control and scheduling of tasks has been proposed in several papers, see e.g.,~\cite{cervin2}. The superiority of evenly spaced sampling in the context of continuous-time output feedback sampled linear systems has been established in~\cite{mirkin1},  both in the $H_2$ and $H_\infty$ senses, using arguments based on the Youla parameterization.  The results in~\cite{mirkin3}, a discrete-time extension of~\cite{mirkin1} considering the $h_\infty$ setting,  do not explicitly handle evenly spaced sampling. In the $h_2$ setting, the impact of the variability of the sampling sequence has been studied in~\cite{mirkin2}, which implies the optimality of evenly spaced sampling for sampled data systems. However, note that our results are different from the ones in~\cite{mirkin3},~\cite{mirkin2}. In particular, in the $h_2$ setting,  neither  the above mentioned convexity properties nor the comparison between arbitrarily schedules differing by two intervals  appear in~\cite{mirkin2}. In particular, while~\cite{mirkin2} shows that the $h_2$ norm is not necessarily a monotone function of the variance of the sampling intervals, from the properties established here we can provide a simple method to find sampling schedules that monotonically improve the $h_2$ norm (see Remark~\ref{rem:final} below).  In the $h_\infty$ sense besides providing a discrete-time result analogous to the continuous-time provided in~\cite{mirkin1}, we address some of its implications not addressed in the literature. Besides the tools we use to derive our results are very different from the tools used in~\cite{mirkin1},~\cite{mirkin2},~\cite{mirkin3}. While a convexity property continuous-time version for the $h_2$ periodic control cost has been established in~\cite{behnam:18}, the property is different from the one given here as will be clarified in the sequel.  The paper~\cite{antunes:2023} goes beyond the present case for the $h_\infty$ problem and searches for state-dependent (event-triggered) scheduling policies that can outperform periodic control. Some of the results in the present paper, which considers only the periodic time-triggered case, are used in~\cite{antunes:2023}.  There are also results that advocate the use of aperiodic sampling, such as~\cite{hadi:19} and~\cite{bini}.  In~\cite{hadi:19} an aperiodic sampling scheme is proposed that guarantees at least the same attenuation bound of evenly-spaced sampling for \textit{ a finite horizon} problem. However, this aperiodic sampling scheme is shown to converge to periodic control (see~\cite[Lemma 1]{hadi:19}), so that there is no contradiction with the result presented here. See~\cite{mirkin1} for an explanation of why the results in~\cite{bini} do not contradict the superiority of evenly spaced sampling  in the average cost sense considered also here.

\par The paper is organized as follows. Section~\ref{sec:2} states the problem and Sections~\ref{sec:3},~\ref{sec:4} provide the main results pertaining to $h_2$ and $h_\infty$ respectively.  Section~\ref{sec:5} provides numerical examples and Section~\ref{sec:6} gives concluding remarks. The proofs of some auxiliary results are given in the appendix.
\section{Problem formulation and Problem Statement}\label{sec:2}
\par Consider a linear system 
\begin{equation}\label{eq:sys}
	x_{t+1} = Ax_t+B_2u_t+B_1w_t, \ \ t\in \mathbb{N}_0:=\mathbb{N} \cup \{0\},
\end{equation}
with the following output of interest
\begin{equation}\label{eq:output}
	z_{t} = C_2 x_t+D_{21}u_t,
\end{equation}
where $x_t \in \mathbb{R}^{n_x}$, $u_t \in \mathbb{R}^{n_u}$, $z_t \in \mathbb{R}^p$, $w_t \in \mathbb{R}^{n_w}$ for $t\in \mathbb{N}_0$. Without loss of generality, we assume that  $C_2^\top D_{21}=0$, and define $Q:=C_2^\top C_2$ and $R = D_{21}^\top D_{21}$. Futhermore, we consider $B_1=I$ and use the notation $B=B_2$. We assume that 
	$(A,B)$ is controllable, $(A,C_2)$ is observable and $R>0$.  
\par We assume that the sensors provide the full state $x_t$. However, not necessarily at every time $t$. In fact, we assume the following measured output
$$ y_t = \left\{\begin{aligned}
	& x_t, \text{ if }\sigma_t=1, \\ 
	& \emptyset \text{ if }\sigma_t=0, \\
\end{aligned}\right.$$
where $\sigma_t$ is a \textit{periodic} binary function with period $h\in\mathbb{N}$ and $y_t=\emptyset$ means that the state is not available at time $t$.  When
\begin{equation}\label{eq:sch}
	\sigma_t=\left\{\begin{aligned}
		&1 \text{ if }t \text{ is zero or an integer multiple of }h,\\
		&0 \text{ otherwise. }
	\end{aligned}\right.
\end{equation}
we have \textit{evenly spaced sampling}. Let  $s_\ell$ be the sampling times defined by $s_{\ell+1}=s_\ell+\tau_\ell$, $s_0=0$ with 
$\tau_\ell = \min\{j \in \mathbb{N}|\sigma_{s_\ell+j}=1\}$. With a periodic schedule, the sampling intervals  $\tau_\ell$ eventually repeat themselves, i.e.,  $\tau_\ell$ = $\tau_{\ell+j}$ for some $j>0$. Note that the first sampling intervals $\mathcal{T}:=(\tau_0,\dots,\tau_{j-1})$, characterize the periodic schedule. Let the average sampling interval be denoted by $h= \frac{1}{j}(\sum_{\ell=0}^{j-1}\tau_\ell)$ and the average rate be denoted by 
$ r=\frac{1}{h}$.  Note that both $r$ and $h$ are rational numbers.  Examples of periodic sampling schedules with the same average rate  are $(\sigma_0,\sigma_1,\sigma_2,\dots,) = (1,0,0,1,0,0,1,0,0,1,0,0,\dots)$ and  $(\sigma_0,\sigma_1,\sigma_2,\dots,) = (1,0,1,0,0,0,1,0,1,0,0,0,\dots)$. The control input $u_t$ can be updated at every time $t$ as a function of the information set $ \mathcal{J}_t = \{x_r|r\in \{0,\dots,t\}, \sigma_r=1\}$, 
that is,  $ u_t = \mu_{u,t}(\mathcal{J}_t)$ for some functions $\mu_{u,t}$.  Performance is measured by either the $h_2$ norm or the $h_\infty$ norm.
\par The $h_2$ norm is defined as follows. Assume that $\{w_k|k\in \mathbb{N}\cup\{0\}\}$ is a sequence of zero-mean independent and identically distributed random variables with $\mathbb{E}[w_kw_k^\intercal]=W\geq 0$. Then the $h_2$ norm is defined as the average cost 
\begin{equation}\label{eq:sys3} J_2=\lim_{T\rightarrow \infty} \frac{1}{T}\mathbb{E}[\sum_{t=0}^{T-1}\|z_t\|^2]\end{equation}
We use the notation $J_{2,h}$ to denote $J_2$ for evenly spaced sampling schedules~\eqref{eq:sch}.

 To define the $h_\infty$ norm let  $w=(w_0,w_1,w_2,\dots)$, $z=(z_0,z_1,z_2,\dots)$,  define the inner product $\langle w,z\rangle=\sum_{t=0}^\infty w_t^\top  z_t$ and norm
$\|w\|^2:=\sqrt{\langle w,w\rangle}$, and let $\ell_2$ be the Hilbert space of sequences with bounded norm. The system provides an attenuation bound $\gamma$ from the input disturbances to the output of interest if 
\begin{equation}\label{eq:inequality}
	\|z\|^2\leq \gamma^2\|w\|^2, \ \ \forall w \in \ell_2, \text{ for } x_0=0.\ \ \  
\end{equation}
We are interested in ensuring that~\eqref{eq:inequality} holds for the smallest possible $\gamma$. The initial condition $x_0$ may be non-zero provided that we redefine~\eqref{eq:inequality} along the lines discussed, e.g., in~\cite{basar:91}. The disturbances depend on the information set  $ \mathcal{I}_t = \{x_r|r\in \{0,\dots,t\}\}$
that is $ w_t = \mu_{w,t}(\mathcal{I}_t)$, for some functions $ \mu_{w,t}$.  We define $\pi_u = (\mu_{u,0}, \mu_{u,1},\dots)$ as the policy of the controller and $\pi_w = (\mu_{w,0}, \mu_{w,1},\dots)$ as the policy of the disturbances. Then, for a given periodic sampling sequence characterized by $\mathcal{T}$, the $h_\infty$ norm coincides with the smallest attenuation bound and is given by 
\begin{equation}\label{eq:gammah_3}	\begin{aligned}
		\gamma_{\mathcal{T}}:= \inf\{\gamma| \exists \pi_u & \text{ such that }\eqref{eq:inequality}\text{ holds when the scheduler} \\ & \text{ is $h$-periodic with sampling intervals $\mathcal{T}$}\}.
	\end{aligned} 
\end{equation} 
For evenly spaced sampling~\eqref{eq:sch} we use the alternative notation
\begin{equation}\label{eq:gammah_}	\begin{aligned}
		\gamma_h:= \inf\{\gamma| \exists \pi_u & \text{ such that }\eqref{eq:inequality}\text{ holds} \\ & \text{ when the scheduler is given by~\eqref{eq:sch}}\}.
	\end{aligned} 
\end{equation}
\noindent Naturally if $\tau_i=h$ for every $i\in \{0,\dots,j-1\}$, we have 
$ \gamma_h = 	\gamma_{\mathcal{T}}.$
\par A sampling schedule characterized by $\tau^1_0,\dots,\tau^1_{j_1-1}$ for some $j_1$ is said to be (strictly) superior to another sampling schedule $\tau^2_0,\dots,\tau^2_{j_2-1}$  in the $h_2$ ($h_\infty$) sense if the corresponding optimal controller achieves a non-larger (strictly smaller) $h_2$ ($h_\infty$) cost. It is said to be optimal if there is not a different strictly superior schedule.
 \par We are interested in finding optimal sampling schedules with average sampling interval $\frac{h}{m}$ in the $h_2$ and $h_\infty$ sense.

\begin{figure}
	\centering
	\includegraphics[width=6cm]{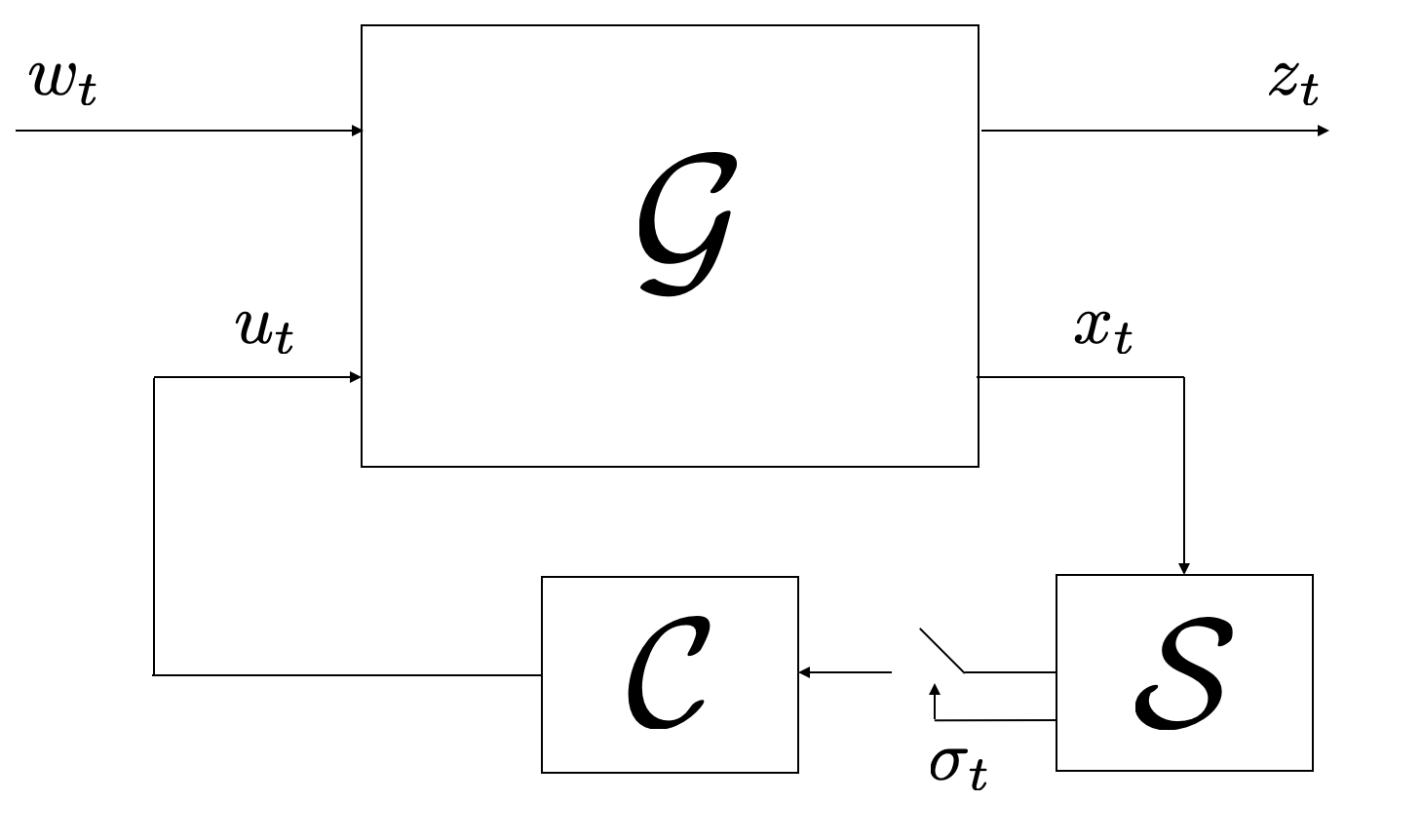}
	\caption{A non-standard sampled feedback ystem consists of a controller $\mathcal{C}$ and a periodic scheduler $\mathcal{S}$, which sends measurement/state data to the controller;  $\mathcal{G}$ represents the plant.}\label{fig:1}
\end{figure} 

\section{Main results for $h_2$ control}\label{sec:3}
\par The following standard result shows that the optimal $h_2$ norm, associated with the optimal controller can be written in terms of a key function $\beta$. Let $P$ be the unique positive definite (since $Q$ is positive definite) solution to
\begin{equation*}\label{eq:riccati}
	\begin{aligned}
		P&=A^\intercal PA +Q -A^\intercal PB(R+B^\intercal PB)^{-1}B^\intercal PA.
	\end{aligned}
\end{equation*}
and let 	$Z=A^\intercal PB(B^\intercal P B+R)^{-1}B^\intercal PA$ and $K=-(B^\intercal P B+R)^{-1}B^\intercal PA$. Let $\text{tr}(X)$ denote the trace of a matrix $X$.
\begin{proposition}\label{prop:1}
		The optimal cost~\eqref{eq:sys3} is given by
		\begin{equation}\label{eq:costper}
		J_{2}=	\text{tr}(PW)+\frac{1}{h}(\sum_{\ell=0}^{j-1} \beta(\tau_\ell))
		\end{equation}
		where $\beta(1)=0$ and, for $p>1$,
		\begin{equation}\label{eq:g}
			\beta(p) = \text{tr}\big(Z (\sum_{s=1}^{p-1}Y(s))\big),
		\end{equation} 
		where \mbox{$Y(s) = \sum_{r=0}^{s-1} A^r W A^\intercal{}^r $}, for \mbox{$s\in \mathbb{N}$}. Moreover, an optimal control policy that minimizes~\eqref{eq:sys3}  is
		\begin{equation}\label{eq:controller1}
			\begin{aligned} u_k = K \hat{x}_{k|k}\end{aligned}
		\end{equation}
		where
		$$ \hat{x}_{k+1|k} = A \hat{x}_{k|k} + Bu_k$$
		\begin{equation*}\label{eq:controller2} \hat{x}_{k|k} = \left\{\begin{aligned}
				   &  \hat{x}_{k|k-1}\text{, if }\sigma_k=1,  \\ 
				   &  \hat{x}_{k|k-1}\text{, if }\sigma_k=0.
			\end{aligned}\right.
		\end{equation*}
		\begin{flushright}
			$\square$
		\end{flushright}
\end{proposition}
\par The proof is given in the appendix.

\par A function with discrete domain $f(i)$, $i\in\mathbb{N}$, is said to be convex if 
\begin{equation}
	f(i)\leq \frac{f(i+1)+f(i-1))}{2}, \ \ \  \forall i\in\mathbb{N}\setminus\{1\}.
\end{equation}
\par A key result of the present paper is the following observation that $\beta$ is convex.

\begin{theorem}
	The function $\beta(i)$, $i\in\mathbb{N}$, defined in~\eqref{eq:g}, is convex.
\end{theorem}
\begin{proof}
	We have, for $i\in\mathbb{N}$,	$\beta(i+1)-\beta(i) = \text{tr}(ZY(i))$	and, for $i\in\mathbb{N}\setminus\{1\}$
	$$(\beta(i+1)-\beta(i))-(\beta(i)-\beta(i-1)) =\text{tr}(A^iWA^\intercal{}^i) \geq 0. $$
\end{proof}
\par The convexity of $\beta$ is used in the next theorem to improve upon a periodic schedule, by replacing two of its sampling intervals by two alternative sampling intervals that are closer to their means.

\begin{theorem}\label{th:keyh2}
	Consider a given periodic schedule with period $h$ and characterized by $j$ intervals $\mathcal{T}= (\tau_0,\dots,\tau_{j-1})$. Given any two distinct $\tau_{i}>\tau_{l}$ construct a modified schedule  $\bar{\mathcal{T}}= (\bar{\tau}_0,\dots,\bar{\tau}_{j-1})$ with 
	$$\bar{\tau}_{\ell}= \left\{\begin{aligned} &
	\tau_{i}-p, \text{ if }\ell=i \\ &  \tau_{j}+p, \text{ if }\ell=l  \\ & \tau_\ell, \text{ otherwise}.	\end{aligned}\right.$$
 for some $p \in \{0,\dots,\lfloor (\tau_{i}-\tau_{l})/2 \rfloor\}$.	Then, $\bar{\mathcal{T}}$ is superior to $\mathcal{T}$ in the $h_2$ sense. 
\end{theorem}
\begin{proof} Let $J_{2}$ and $\bar{J}_{2}$ denote the optimal costs associated with the original and modified schedules, respectively. Note that both schedules have the same period denoted by $h$. Due to~\eqref{eq:costper} is suffices to prove that
\begin{equation}\label{eq:aa}J_{2}-\bar{J}_{2} = \frac{1}{h}(\beta(\tau_i)+\beta(\tau_l)-(\beta({\tau}_i-p)+\beta({\tau}_l+p)) )\geq 0.\end{equation}
Since $\beta(i)$ is convex we have, for any $p\in \mathbb{N}$,
\begin{equation}
\begin{aligned}
	&	\beta(i+p+1)-\beta(i+p)\geq \beta(i+p)-\beta(i+p-1)\geq \dots\\
	& \ \ \ \ \ \ \ \ \ \ \ \ \ \ \ \ \ \ \  \geq \beta(i+1)-\beta(i)\geq 0, \ \ \ \ \ \ \forall i\in\mathbb{N}.
\end{aligned}
\end{equation}
\par From this fact we conclude that
\begin{equation}
	\begin{aligned}
			\beta(\tau_i)-\beta(\tau_i
		-p)&= \sum_{k=0}^{p-1} \beta(\tau_i-k)-\beta(\tau_i
		-1-k)\\
	& \geq 	\sum_{k=0}^{p-1} \beta(\tau_l+p-k)-\beta(\tau_l
		+p-1-k)\\
		&  = \beta(\tau_l+p)-\beta(\tau_l)
	\end{aligned}
\end{equation}
which implies~\eqref{eq:aa} concluding the proof.
\end{proof}

\par This results has several implications given next.
\begin{corollary}\label{th:4}
An $h$-periodic schedule with $m$ intervals between sampling is optimal if $m_1$ of these intervals equal $h_1=\lfloor \frac{h}{m}\rfloor$ and $m_2$ equal $h_2=\lceil \frac{h}{m}\rceil$ where $m_1+m_2=m$ $h_1m_1+h_2m_2=h$.  Moreover, the corresponding $h_2$ norm is equal to
\begin{equation}\label{eq:Fb}
	\frac{m_1}{h} \beta(h_1) + \frac{m_2}{h}\beta(h_2)
\end{equation} 
Furthermore, if $\beta(i+1)>\beta(i)$ for every $i \in \mathbb{N}$, then  these optimal schedules are unique in the class of schedules with period $h$.
\end{corollary}
\begin{proof} We can  list all possible schedules with period $h$ and $m$ sampling intervals and compute the associated $h_2$ norm. Note that all the schedules that meet the form in the present corollary have the same $h_2$ norm due to~\eqref{eq:costper}. If such an $h_2$ norm is minimal we conclude the sufficiency part. In turn, if we would have a schedule with minimal $h_2$ norm that does not take the form stated in the present corollary, due to Theorem~\ref{th:keyh2}, we could modify it without increasing the cost so that it does meet the mentioned form, so that it is equally optimal.   If $\beta(i+1)>\beta(i)$ for every $i \in \mathbb{N}$ this leads to a contradiction meaning that only the schedules of the form stated in the present corollary are optimal.
\end{proof}
\par It is immediate from~\eqref{eq:g} that a sufficient condition for $\beta(i+1)>\beta(i)$ for every $i \in \mathbb{N}$ is $KW\neq0$.
\par Corollary~\ref{th:4} implies the following:
\begin{itemize}
	\item[i)] In general there might be more than one optimal schedule, e.g., if $h=10$, $m=4$, $1010100100$ are $1001010010$ are both optimal. 
		\item[ii)]Given a desired rational average sampling time $\frac{h}{m}$, the optimal schedules are the ones that have $m_1$  intervals equal to $\lfloor \frac{h}{m}\rfloor$ and $m_2$ equal to $\lceil \frac{h}{m}\rceil$. 
		\item[(iii)] Note that we can write~\eqref{eq:Fb} as follows
		$$\frac{\beta(h_1)(h_1+1-\zeta) +\beta(h_1+1)(\zeta-h_1)}{\zeta}|_{\zeta=\frac{h}{m}}$$ 
		since $h = m_1h_1+m_2h_1+m_2$ and
		$$\begin{aligned}
			& \frac{(h_1+1-h/m)}{h/m} = \frac{\overbrace{h_1m_1+h_1m_2+m_2-h}^{=0}+m_1}{h}\\
			& \frac{(h/m-h_1)}{h/m} = \frac{\overbrace{h-h_1m_1-h_2m_2}^{m_2}}{h}
		\end{aligned}  $$
	 Rewriting~\eqref{eq:Fb}  in terms of the rate $r=1/\xi$ we have
	$$\beta(h_1)(r(h_1+1)-1) +\beta(h_1+1)(1-rh_1)|_{r=\frac{m}{h}}$$ 
		Thus, the plot optimal achievable $h_2$ norm versus  average rate is (a restriction to the rational numbers of) a continuous piecewise affine function connecting the pairs $(1/h,J_{2,h})$ where $J_{2,h}$ is the $h_2$ norm of periodic control with integer average sampling interval $h$ (see Figure~\ref{fig:3} below).
\end{itemize} 
\par Corollary~\ref{th:4} also implies the following result.
\begin{corollary}
		Evenly spaced sampling is optimal in the class of periodic schedulers with integer average sampling time $h\in\mathbb{N}$. This is the unique optimal schedule if $g(h-1)<g(h)<g(h+1)$ when $h>1$ and if $g(1)<g(2)$ when $h=1$.
		\begin{flushright}
			$\square$
		\end{flushright}
\end{corollary}

\begin{remark}\label{rem:final}
	One simple method to arrive at an optimal sampling schedule as given in Corollary~\ref{th:4} from an arbitrary scheduled is to recursively reduce the largest interval by the same amount that the smallest is increased, and set this amount to the largest possible according to Theorem~\ref{th:keyh2}. The resulting sequence of sampling schedules monotonically improves the $h_2$ norm.
\end{remark}

\begin{remark}\label{rem:1}
\par The paper~\cite{antunes:16_1} considers a continous-time version of sampled data periodic control with average intersampling time $h_c\in\mathbb{R}_{>0}$. This reference provides the following expression for the cost of periodic control with average sampling period $h_c$:
\begin{equation}\label{eq:F1}
	J_c(h_c)=\delta_c+\frac{1}{h_c}\beta_c(h_c),\, \beta_c(h_c) \!=\! \int_0^{h_c}\!\!\text{tr}(Z_c\int_0^s \!V(r)dr)ds
\end{equation} 
where $ V(s)=\int_0^s e^{A_cr}W_ce^{A_c^\intercal r}$, $A_c$ is the system matrix, $W_c$ is a positive semi-definite matrix proportional to the covariance of the stochastic disturbances and the expressions for $\delta_s$ and $Z_c$ are omitted here.  Note that analogously to the discrete-time case $\beta_c$ is a convex function since $\beta_c''(h_c)=\text{tr}(Z_cV(h_c))\geq 0$. Due to this convexity property, many results of the previous section can be extended to the sampled-data case; however we do not pursue this further here.
\end{remark} 
\begin{remark}\label{rem:2}
	A different result, shown in~\cite{behnam:17}, states that if we write~\eqref{eq:F1} as a function of the average rate $r = 1/h$ such a function, $J_c(1/r) = r\beta_c(\frac{1}{r})$ is convex, Note,  however, that neither $J_c(h)$ nor $J_2(h)$ are in general convex functions of $h$.
\end{remark} 


\section{Main results for $h_\infty$ control}\label{sec:4}

\par We start by considering evenly spaced sampling~\eqref{eq:sch} and by providing a method to compute $\gamma_h$, given by~\eqref{eq:gammah_}. Let us first define three matrix transformations:
$$ \begin{aligned}
	F_a(P)&:= P+P(\gamma^2 I-P)^{-1}P \\
	F_c(P)&:= A^\top PA+Q-A^\top PB(B^\top PB+R)^{-1}B^\top PA  \\
	F_o(P)&:= A^\top PA+Q.
\end{aligned}$$
for a given $\gamma \in \mathbb{R}_{>0} $.  When $h=1$ the following iteration 	$P_{t+1}=F_c(F_a(P_t))$ 
with $P_0=0$ is monotone in the sense that $P_{t+1}\geq P_t$ converges if $\gamma^2 I>P_t$ for every $t \in \mathbb{N}_0$ to the unique positive definite solution $\bar{P}_\gamma$ of the algebraic Riccati equation 
\begin{equation}\label{eq:Pgamma} 	\bar{P}_\gamma=F_c(F_a(\bar{P}_\gamma)), \end{equation}
see ~\cite{basar:91} (although the expressions in~\cite{basar:91} appear in a different but equivalent form). Due to monotonicity, $\gamma^2 I>\bar{P}_\gamma$ implies that  $\gamma^2 I>P_t$ for every $t \in \mathbb{N}_0$. Provided that this condition holds,~\eqref{eq:inequality} holds for a policy $\pi_x$ specified by $u_t= K_\gamma x_t$, 
where 
\begin{equation}\label{eq:K}
	K_\gamma= -(R+B^\top F_a(\bar{P}_\gamma)B)^{-1}B^\top F_a(\bar{P}_\gamma)A.
\end{equation}
%
If $\gamma$ is such that $\gamma^2 I \geq P_{t}$ does not hold for some $t$, then~\eqref{eq:inequality} does not hold for any $\pi_u$.
\par However, if $h>1$ the conditions on $\gamma$ for the existence of such a control policy are stricter, i.e., $\gamma$ needs to be larger~\cite{hadi:20}. They actually become stricter as $h$ increases leading to non-decreasing sequence of $\gamma_h$, as stated in the next lemma.
\begin{lemma}\label{lem:1}
	Suppose that $\gamma^2 I>\bar{P}_\gamma$, and consider the following iteration, with $M_1=\bar{P}_\gamma$,
	\begin{equation}\label{eq:iter}
		M_{k+1} = F_o(F_a(M_k)), \ \ \ k \in \{1,2,\dots,{h}\},
	\end{equation} 
	which can be ran as long as $\gamma^2 I-M_k$ is not singular. Then:
	\begin{itemize}
		\item[(i)]	if $\gamma^2 I\!-\!M_{k}\!>\!0$, for  all $k \!\in\! \{1,\dots,h\}$, then, for all  $k \!\in\! \{1,\dots,h-1\}$, $M_{k+1}\!\geq\! M_k$. 
		\item[(ii)] for every $h \in \mathbb{N}$, $\inf\{ \gamma|\gamma^2 I-M_h>0\}=\gamma_h$, where $\gamma_h$ is given by~\eqref{eq:gammah_}.  
		\item[(iii)] for every $h \in \mathbb{N}$, $\gamma_{h+1}\geq \gamma_h$.\hfill	$\square$
	\end{itemize}
\end{lemma}
\par The proof is given in the appendix.
\par We turn now to general schedules. The following result provides a simple way of computing the $h_\infty$ norm associated with a general schedule from the $h_\infty$ norm associated with an evenly spaced schedule. 
\par \begin{theorem}\label{th:5}
	Consider a periodic schedule characterized by the intervals $\mathcal{T}:=(\tau_0,\dots,\tau_{j-1})$. Let $\bar{\tau}=\max\{\tau_i|i\in \{0,1,\dots,j-1\}\}$. Then,
	$$ \gamma_{\mathcal{T}}=\gamma_{\bar{\tau}}.$$
\end{theorem}
\begin{flushright}
$\square$
\end{flushright}
\begin{proof}  Suppose that 	$\gamma_{\mathcal{T}}>\gamma_{\bar{\tau}}$. 
		Pick a $\gamma$ such that
		$ \gamma_{\mathcal{T}}>\gamma>\gamma_{\bar{\tau}}$. Then for any given arbitrary control policy and for schedules $\mathcal{T}$ there exists a $w\in \ell_2$ such that $ \|z\|^2-\gamma^2\|w\|^2>0$
		i.e., 
		$ \lim_{T\rightarrow \infty} \sum_{t=0}^{T-1}z_t^\intercal z_t-\gamma^2w_t^\intercal w_t>0$. This implies that there exists a  $\underline{t}\in \mathbb{N}$ and a $w\in \ell_2$ such that, for $\bar{t}\geq \underline{t}$,
		\begin{equation}\label{eq:a}
			 \sum_{t=0}^{\bar{t}-1}z_t^\intercal z_t-\gamma^2w_t^\intercal w_t>0
		\end{equation} 
		Suppose that we pick $\bar{t}$ to be a multiple of $h$. Using  Lemma~\ref{lem:2} in the appendix, we conclude that if we pick the control policy $u_t = K_tx_t$, $t\in \{0,1,\dots,\bar{\tau}\}$ where the $K_t$ are obtained from the iteration~\eqref{eq:iteration} initialized with $P_\tau=Y_0=0$ and $\tau=\bar{t}$, we get
		
	{\footnotesize
	\begin{equation}\label{eq:FF}
		\begin{aligned}
			&\sum_{k=0}^{\bar{t}-1}z_k^\top z_k-\gamma^2 w_k^\top w_k =\underbrace{x_0}_{=0}^\top P_0 x_0-\\\
			&\sum_{k=0}^{\bar{t}-1}(w_k-{L}_k(Ax_k+Bu_k))^\top (\gamma^2 I-P_{k+1})(w_k-{L}_k(Ax_k+Bu_k))
\end{aligned}\end{equation}	}
provided that the $\gamma^2I-P_t$ are invertible. This is indeed the case since as we now argue $\gamma^2I-P_t>0$ for every $t \in \{0,1,\dots,\bar{t}-1\}$. To see this it suffices to establish that $P_t\leq \bar{P}_\gamma$ for every $t \in \{0,1,\dots,\bar{t}-1\}$ since $\gamma^2I-\bar{P}_\gamma>0$ by hypothesis. If we run the iteration~\eqref{eq:iteration} for $\tau=\bar{t}$ with $Y_0 = \bar{P}_\gamma$ we obtain $P_t = P_\gamma$ for every $t \in \{0,1,\dots,\bar{t}\}$. From the monotonicity property of Lemma~\ref{lem:1} we conclude that the $P_t$ obtained when $Y_0=0$ satisfy $P_t\leq \bar{P}_\gamma$ for every $t \in \{0,1,\dots,\bar{t}-1\}$ as desired. Note that~\eqref{eq:FF} implies $\sum_{t=0}^{\bar{t}-1}z_t^\intercal z_t-\gamma^2w_t^\intercal w_t\leq 0$ for every disturbance sequence which contradicts~\eqref{eq:a}.
	\par	Suppose now that 
			$$ \gamma_{\mathcal{T}}<\gamma_{\bar{\tau}}$$
		so that for the schedule $\mathcal{T}$ we can guarantee that
	\begin{equation}\label{eq:zw}
		\|z\|^2-\gamma^2\|w\|^2<0
	\end{equation} 
		for some $\gamma$, $\gamma_{\mathcal{T}}<\gamma<\gamma_{\bar{\tau}}$ and for every $w\in \ell^2$. Note that in this case $\gamma^2I-M_{\bar{\tau}}$ has an eigenvalue which is negative. Let $\ell_m \in \text{arg}\max \{\tau_\ell|\ell \in \{0,\dots,j-1\}\}$ so that $\tau_{\ell_m}=\bar{\tau}$ and consider the following disturbance policy
		$$ w_t = \left\{\begin{aligned}
			& \xi ,\ \ t \in \{0,\dots,s_{\ell_m}-1\} \\ 
		    & \tilde{L}_{\bar{\tau} } (Ax_t+Bu_t)+\eta\text{ if }t=s_{\ell_m} \\
		    & \tilde{L}_{\bar{\tau}-t-s_{\ell_m}} \! \!(Ax_t\!+\!Bu_t),\! \!\text{ if }t \!\in\! \{s_{\ell_m}\!+\!1,\dots,s_{\ell_m}\!+\!\bar{\tau}\!-\!1\! \} \ \\
		    & \bar{L}_t (Ax_t+Bu_t) \ \ t \in \{s_{\ell_m}+\bar{\tau},\dots,s_{\ell_m}+\bar{\tau}+q \} \\
		    & 0, \text{ if }t\geq s_{\ell_m+1}+q+1\\
		\end{aligned}\right.$$
	where:
	\begin{itemize}
		\item $\xi \in \mathbb{R}^{n_x}$ is an arbitrary constant;
		\item $\eta \in \mathbb{R}^{n_x}$ will be chosen in the sequel;
		\item $\tilde{L}_k  =  (\gamma^2 I-M_{k})^{-1}M_{k} $, $k\in \{1,\dots,\bar{\tau}-1\}$;
		\item $\bar{L}_t$, $t \in \{s_{\ell_m}+\bar{\tau},\dots,s_{\ell_m}+\bar{\tau}+q\}$ are  given by ~\eqref{eq:LK} in Lemma~\ref{lem:4} below (with $k=s_{\ell_m}+\bar{\tau}$)  and $q$ is such that $\|x_{s_{\ell_m}+\bar{\tau}}^\intercal G_q x_{s_{\ell_m}+\bar{\tau}}-x_{s_{\ell_m}+\bar{\tau}}^\intercal \bar{P}_\gamma x_{s_{\ell_m}+\bar{\tau}}|<\alpha$ for a given and arbitrary $\alpha>0$. Such a $q$ depends on $x_{s_{\ell_m}+\bar{\tau}}$ and exists due to Lemma~\ref{lem:4}.
	\end{itemize}
Then
	 \begin{eqnarray}
&&	\!\!\!\!\!\! \|z\|^2-\gamma^2\|w\|^2 \nonumber \\
&&  \!\!\!\!\!\! =\underbrace{\sum_{t=0}^{{s}_{\ell_m}-1}z_t^\intercal z_t-\gamma^2w_t^\intercal w_t}_{:=c}+\sum_{t={s}_{\ell_m}}^{{s}_{\ell_m}+\bar{\tau} -1}z_t^\intercal z_t-\gamma^2w_t^\intercal w_t \nonumber\\
&& \ \ \ \ \ \ \ \ \ \ \ \ \ \ \ \ \ \ \ \ +\underbrace{\sum_{t=\bar{s}_{\ell_m}+\bar{\tau}}^\infty z_t^\intercal z_t-\gamma^2w_t^\intercal w_t}_{\geq x_{s_{\ell_m}+\bar{\tau}}^\intercal G_q x_{s_{\ell_m}} \text{  due to~\eqref{eq:ine} below }} \nonumber \\
&& \!\!\!\!\!\!\geq c+\sum_{t={s}_{\ell_m}}^{{s}_{\ell_m}+\bar{\tau} -1}z_t^\intercal z_t-\gamma^2w_t^\intercal w_t + x_{{s}_{\ell_m}+\bar{\tau}}^\intercal \bar{P}_\tau x_{{s}_{\ell_m}+\bar{\tau}} \nonumber \\
&& \ \ \ \ \ \ \ \ \ \ \ \ \underbrace{-x_{{s}_{\ell_m}+\bar{\tau}}^\intercal \bar{P}_\tau x_{{s}_{\ell_m}+\bar{\tau}}+x_{{s}_{\ell_m}+\bar{\tau}}^\intercal G_q x_{{s}_{\ell_m}+\bar{\tau}}}_{\geq \alpha} \nonumber \\
&& \!\!\!\!\!\!\geq c+\alpha+x_{{s}_{\ell_m}}^\intercal \bar{P}_\tau x_{{s}_{\ell_m}}- \nonumber\\
&&\!\!\!\!\!\!\sum_{k=s_{\ell_m}}^{s_{\ell_m}+\bar{\tau}-1}\!\!\!\!(w_k-\tilde{w}_k)^\top (\gamma^2 I-M_{s_{\ell_m}+\bar{\tau}-k})(w_k-\tilde{w}_k))\label{eq:Fa}\\
 && \!\!\!\!\!\!= c+\alpha + x_{{s}_{\ell_m}}^\intercal \bar{P}_\tau x_{{s}_{\ell_m}}-\eta^\intercal (\gamma^2I-M_{\bar{\tau}})\eta > 0\nonumber
	\end{eqnarray}
where $\tilde{w}_k=\tilde{L}_{\bar{\tau}-k-s_{\ell_m}} \! \!(Ax_k\!+\!Bu_k)$, in~\eqref{eq:Fa} we have used Lemma~\ref{lem:3} below with $Y_0=P_\tau$ and $\tau = \bar{\tau}$, and in the last inequality we have picked $\eta$ to be aligned with an eigenvector associated with an eigenvalue of $(\gamma^2I-M_{\bar{\tau}})$ that is negative and multiplied by a constant high enough to make the expression positive. This contradicts~\eqref{eq:zw} concluding the proof.
\end{proof}

Theorem~\ref{th:5} has several consequences discussed next.

\begin{corollary}
A sampling schedule guarantees the smallest attenuation bound ($h_\infty$ norm) achievable for a given rational average sampling time $\frac{h}{m}$ if  its largest interval does not exceed $\tilde{h}=\lceil \frac{h}{m}\rceil $. Moreover, if \begin{equation}\label{eq:Fc}
	\gamma_{\tilde{h}-1}>\gamma_{\tilde{h}}>\gamma_{\tilde{h}+1}
\end{equation} the schedules satisfying this property are the unique schedules that guarantee the smallest attenuation bound.    Furthermore, under~\eqref{eq:Fc}, the schedule with smallest average rate that guarantees this attenuation bound corresponds to evenly space sampling. 
\end{corollary}	
\begin{flushright}
	$\square$
\end{flushright}
	 \par Note that also here there can be multiple optimal schedules (in the sense of this corollary), but are in general different from the ones for the $h_2$ case.
	\par Due to this corollary, the plot optimal achievable $h_\infty$ norm versus average sampling time (or average rate) is  (a restriction to the rational numbers of) a discontinuous piecewise constant function, where these constants are equal to the $h_\infty$ norms of evenly spaced sampling with integer average sampling interval $h$ (see Figure~\ref{fig:4} below).
	
	\par Note that we can have schedules corresponding to arbitrarily poor $h_\infty$ norm and maximum sampling rate. In fact,  the schedule characterized by $h=\omega b+b$
	$$ \tau_i= \left\{\begin{aligned}
	&	1 \text{, if }i \in \{0,\dots,\omega b\} \\ 
	& b\text{, if }i=\omega b+1
	\end{aligned}\right.$$
leads to an $h_\infty$ norm $\gamma_{b}$ and average rate $\frac{\omega b}{(\omega +1)b}$. For systems for which $\gamma_{b}\rightarrow \infty$ as $b\rightarrow  \infty$, the $h_\infty$ norm becomes arbitrary poor while the average rate converges to $1$ as $b\rightarrow  \infty$ and $\omega\rightarrow  \infty$.
\section{Numerical example}\label{sec:5}
\par Suppose that 
$$A = \begin{bmatrix}
	1 & 1 & 1 \\ 0 & 1 & 1 \\ 0 & 0 &1 
\end{bmatrix}, \ \ B = \begin{bmatrix}
0 \\
0\\
1  
\end{bmatrix}, \ \ Q = I_3, \ \ R =1, $$
where $I_3$ is the identity matrix. The values of $\beta(h)$, $J_{2,h}$, and $\gamma_h$ are given in Table~\ref{tb:1} and plotted in Figures~\ref{fig:2},~\ref{fig:3},~\ref{fig:4}, respectively.

{\footnotesize
	\begin{table}[b!]
		\centering
\begin{tabular}{c|cccccc}
	$h $ & 1 & 2 & 3  & 4 & 5 & 6  \\
	\hline 
	$\beta(h)$ & 0   &   38.1 & 179.1 & 548.2  &  1361.3 & 2960.2   \\
	\hline 
	$J_{2,h}$ & 14.3 &   33.4  & 73.9 &  151.9  & 286.5 & 507.6  \\
	\hline
	$\gamma_h$ & [3.805 & 7.97 & 14.55 & 24.18  & 37.45 & 54.58  
\end{tabular}
\caption{Values of $\beta(h)$, $J_{2,h}$, and $\gamma_h$ for $h\in \{1,2,\dots,6\}$ for the numerical example.}\label{tb:1}
\end{table}
}

In Figures~\ref{fig:3} the optimal $h_2$  value achievable with the corresponding rational average sampling time is plotted and in Figure~\ref{fig:4} plotted and the optimal $h_\infty$  value achievable with the corresponding rational average rate is plotted.

\begin{figure}[t!]
	\centering
	\includegraphics[width=6cm]{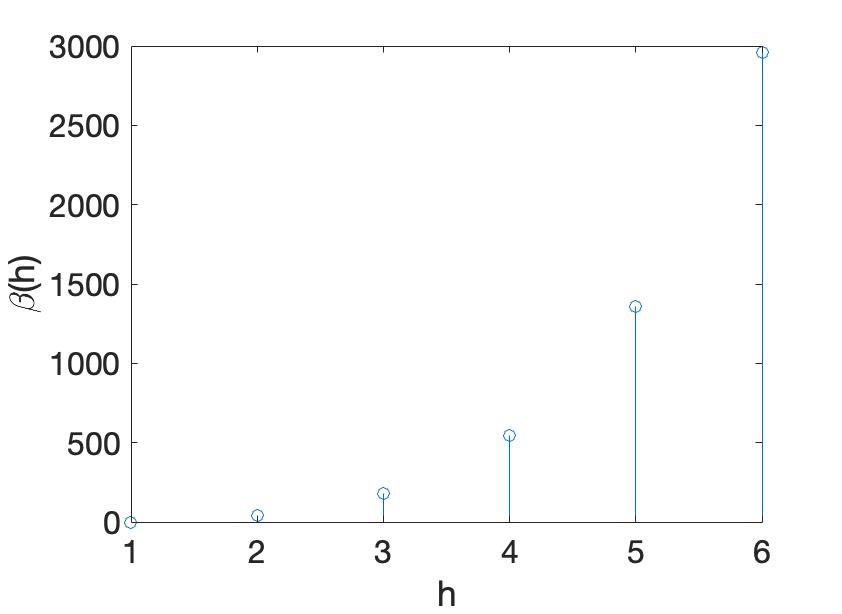}
	\caption{$\beta(h)$ versus $h$}\label{fig:2}
\end{figure} 

\begin{figure}[t!]
	\centering
	\includegraphics[width=6cm]{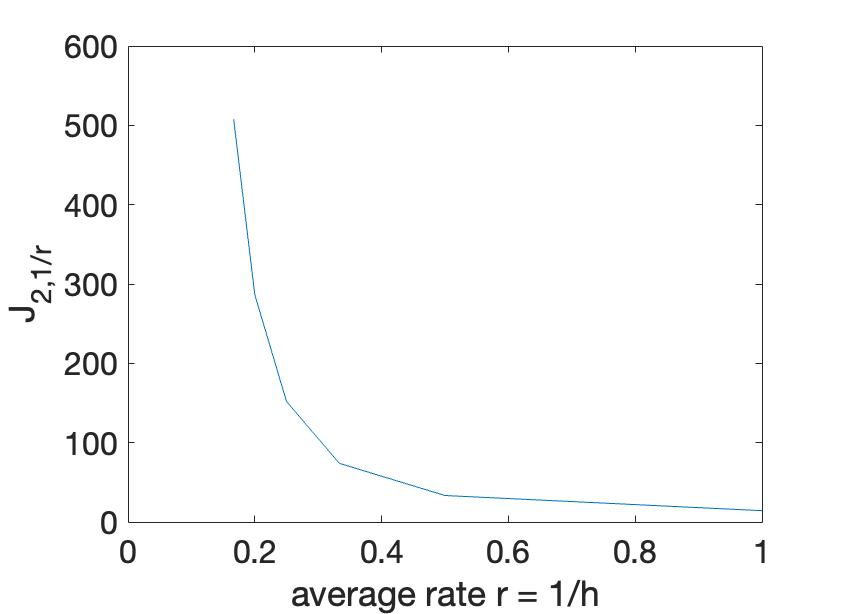}
	\caption{$J_{2,h}$ versus average rate}\label{fig:3}
\end{figure} 

\begin{figure}[t!]
	\centering
	\includegraphics[width=6cm]{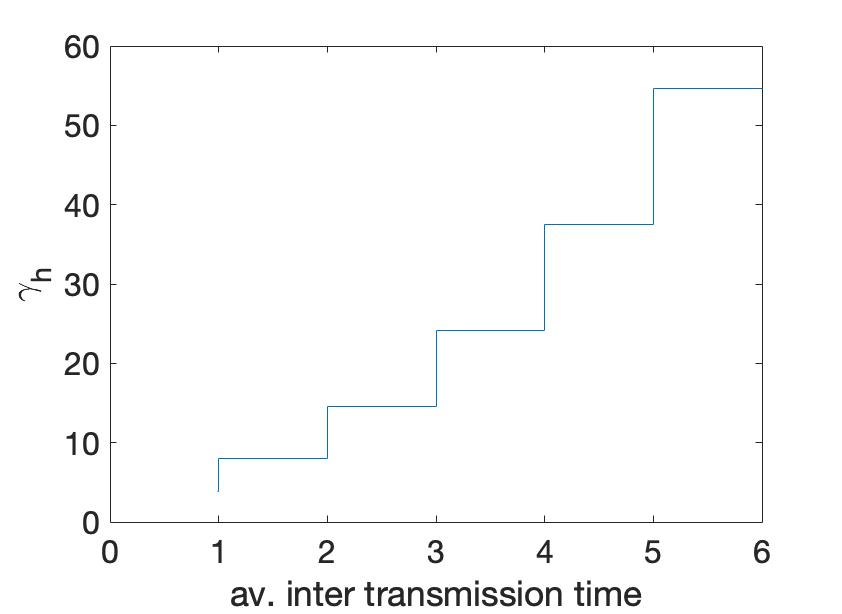}
	\caption{$\gamma_{h}$ versus average sampling interval}\label{fig:4}
\end{figure}

\section{Conclusions}\label{sec:6}
\par In  this paper we have characterized sampling schedules that are optimal for  $h_2$ and $h_\infty$ feedback control. We have shown that if the desired average intersampling time is $h$ then evenly spaced sampling is optimal in both senses. However, when the desired average intersampling time is not an integer, then the class of optimal schedules is different in the $h_2$ and $h_\infty$ senses. While for $h_2$ schedules close to evenly spaced sampling are still optimal, in the $h_\infty$ framework the $h_\infty$ norm is only dictated by the largest sampling interval.

\appendix

\section*{Proof of Proposition~\ref{prop:1}}

Consider the following cost
\begin{equation}\label{eq:costproof}
	\mathbb{E}[\sum_{t=0}^{T-1}\|z_t\|^2+x_T^\intercal P x_T]
\end{equation}
where  $T=mh$ for some $m\in \mathbb{N}$. From standard optimal control results, the optimal policy that minimizes~\eqref{eq:costproof}  is~\eqref{eq:controller1} where $\hat{x}_{k|k}$ is a Kalman filter estimate, and the optimal cost is
\begin{equation}\label{eq:costproof2} x_0^\intercal P x_0+\sum_{t=0}^{T-1}\text{tr}(Z\Phi_t)+T\text{trace}(ZW).
\end{equation}
where $\Phi_t = \mathbb{E}[(x_t-\hat{x}_{t|t})(x_t-\hat{x}_{t|t})^\intercal]$, 
see~\cite[Ch. 5]{bertsekas:05}. Since $x_t-\hat{x}_{t|t}=0$ when $\sigma_t = 1$, $\Phi_t$ resets to $\Phi_t=0$ when $\sigma_t = 1$. Moreover, it is a periodic function equal to 
$$\Phi_t=\left\{\begin{aligned}
	& 0 \text{ if }\sigma_t=1, \\ 
	&Y(t-s_{\bar{\ell}(t)}), \text{ otherwise, }
\end{aligned}\right.$$
where $\bar{\ell}(t)=\max\{\ell|s_\ell\leq t\}$. Then
$$ \lim_{T\rightarrow \infty}\frac{1}{T}\sum_{t=0}^{T-1}\text{tr}(Z\Phi_t) = \frac{1}{h}(\sum_{\ell=0}^{j-1} \gamma(\tau_\ell))$$
Dividing~\eqref{eq:costproof2}  by $T$ and taking the limit as $T\rightarrow \infty$ we obtain the desired conclusion (note that~\eqref{eq:controller1} is the unique optimal policy for~\eqref{eq:costproof2} but not unique for~\eqref{eq:sys3} while still optimal).
\section*{ Proof of Lemma~\ref{lem:1}}\label{sec:A}

\par (i) When $k=1$, $M_1=\bar{P}_\gamma = F_c(F_a(\bar{P}_\gamma)) = M_2-A^\top F_a(\bar{P}_\gamma)B(B^\top F_a(\bar{P}_\gamma)B+R)^{-1}B^\top F_a(\bar{P}_\gamma)A\leq M_2$. We now prove that, $F_o(F_a(P))$ is a monotone  map in the sense that $F_o(F_a(P_1))\leq F_o(F_a(P_2))$when $P_1\leq P_2$. This follows from the fact that, for $P<\gamma^2I$ and that for an arbitrary $x$
$$x^\intercal F_o(F_a(P))x=  x^\intercal Qx+\max_w (Ax+w)^\intercal P(Ax+w)$$
Using induction, suppose that $M_{k-1}\leq M_k$ for some $k \in \{2,\dots,h-1\}$. Then  $F_o(F_a(M_{k-1}))\leq F_o(F_a(M_k))$ and, concluding the proof.
\par (ii) In~\cite{hadi:20} it is shown that $\gamma_h = \inf\{ \gamma|\gamma^2 I-\bar{D}_h^\top \text{diag}(I\otimes Q,\bar{P}_\gamma)\bar{D}_h>0\}$, with
$$\bar{D}_h=\left[\begin{array}{cccc}
	I & 0 & \dots & 0 \\
	A  & I & \dots & 0 \\
	\vdots & \vdots & \vdots & \vdots \\
	A^{\tau-1}  & A^{\tau-2} & \dots & I
\end{array}\right]$$
and the condition $\gamma^2 I-\bar{D}_h^\top \text{diag}(I\otimes Q,\bar{P}_\gamma)\bar{D}_h>0$  is equivalent to the following function being concave in $w_0,\dots, w_{h-1}$:
$$ \sum_{k=0}^{h-1}z_k^\top z_k -\gamma^2 w_k^\top w_k+ x_h^\top \bar{P}_\gamma x_h,$$
where $x_{k+1}=Ax_k+w_k$ for $k\in \{0,1,\dots,h-1\}$. Applying dynamic programming to maximize this function with respect to  $w_{\ell+1},\dots, w_{h-1}$, for $\ell \in \{0,\dots,h-1\}$,  we obtain that this cost is equal to
$$
\sum_{k=0}^{\ell}z_k^\top z_k -\gamma^2 w_k^\top w_k+ x_{\ell+1}^\top M_\ell x_{\ell+1}$$
We can rewrite this expression as
\begin{equation*}\label{eq:cost}
	\begin{aligned}
		& \sum_{k=0}^{\ell-1} z_k^\top  z_k  -\gamma^2 w_k^\top w_k + x_\ell^\top \!(Q\!+\!A^\top M_\ell A )x_\ell \\
		& \ \ \ \ \ \ \ \ \ \ \ \ + w_{\ell}^\top  (M_\ell -\gamma^2I)w_{\ell} +2w_{\ell}^\top M_\ell A x_\ell
	\end{aligned}
\end{equation*} 
from which clear that this function is concave in $w_{\ell} $ if and only if $\gamma^2 I-M_\ell>0$ where $\ell \in \{1,\dots,h\}$ is arbitrary.

(iii) Follows from (i) and (ii).

\section*{Auxiliary Lemmas}
	\begin{lemma}\label{lem:2}
		If $0\leq P_1\leq P_2$ then
		$$ F_c(F_a(P_1))\leq F_c(F_a(P_2))$$
		\begin{flushright}
			$\square$
		\end{flushright}
		\end{lemma}
		\begin{proof}
						The proof follows by noticing that, for every $x\in \mathbb{R}^{n_x}$
			$$ \begin{aligned}
				x^\intercal F_c(F_a(P_1)) x &= \min_u \max_w x^\intercal Qx+u^\intercal Ru-\gamma^2 w^\intercal w+\\
				&(Ax+Bu+w)^\intercal P(Ax+Bu+w)^\intercal.
			\end{aligned}$$
		\end{proof}
	\begin{lemma}\label{lem:3}
		Consider~\eqref{eq:sys},~\eqref{eq:output}, and the following iteration
		\begin{equation}\label{eq:iteration} P_{k-1} = F_c(F_a(P_k)), k \in \{\tau,\tau-1,\dots,0\} \end{equation}
		with $P_\tau =Y_0$ for some  $Y_0\geq 0$ and
		where $\gamma$ is such that $\gamma^2 I-P_k$ is invertible for every  $k \in\{\tau,\tau-1,\dots,0\}$.  Then
		{\footnotesize
			\begin{equation}\label{eq:FF3}
				\begin{aligned}
					&\sum_{k=0}^{\tau-1}z_k^\top z_k-\gamma^2 w_k^\top w_k+x_\tau^\top P_\tau x_\tau=x_0^\top P_0 x_0+\\
					& \sum_{k=0}^{\tau-1}(u_k-K_k x_k)^\top (R+B^\top F_a(P_{k+1})B)(u_k-K_k x_k)-\\
					&\sum_{k=0}^{\tau-1}(w_k-{L}_k(Ax_k+Bu_k))^\top (\gamma^2 I-P_{k+1})(w_k-{L}_k(Ax_k+Bu_k))
		\end{aligned}\end{equation}	}
		where $K_k$ and $L_k$ are given by
		$$\begin{aligned}
			K_k &= -(R+B^\top F_a(P_{k+1})B)^{-1}B^\top F_a({P}_{k+1})A,  \\
			L_{k}&=(\gamma^2I-P_{k+1})^{-1}P_{k+1} A
		\end{aligned}	 $$
		for $k\in\{0,1,\dots,\tau-1\}$ \hfill 	$\square$
	\end{lemma}
	
	\begin{proof2}
		By completion of squares we obtain for every $k \in \{0,\dots,\tau-1\}$
			{	\footnotesize
		\begin{equation*}
			\begin{aligned}
				&-\gamma^2 w_{k}^\top w_{k}+x_{k+1}^\top P_{k+1} x_{k+1}=f_1(x_{k},u_{k},w_{k},k)+\\
				& (Ax_{k}\!+\!Bu_{k})^\top  \underbrace{(P_{k+1} + P_{k+1}(\gamma^2 I-P_{k+1})^{-1} P_{k+1})}_{=F_a(P_{k+1})}(Ax_{k} +Bu_{k})\\
			\end{aligned}
		\end{equation*}
	}
		where
		{	\footnotesize
			\begin{equation*}
				\begin{aligned}
					&	f_1(x,u,w,k)=\\
					& \ \ \ \ -(w-L_k(Ax+Bu)  )^\top  (\gamma^2 I-P_{k+1}) (w-L_k(Ax+Bu)  )
				\end{aligned}
		\end{equation*} }
		Using this equality and again by completion of squares we obtain
			{	\footnotesize
		$$\begin{aligned}
			&z_{k}^\top z_{k}  \!-\! \gamma^2w_{k}^\top w_{k}\!+\!x_{k+1}^\top P_{k+1}x_{k+1} \!=\!f_1(x_{k},u_{k},w_{k},k)\!+\! f_2(x_{k},u_{k},k)\\
			&+x_{k}^\top \underbrace{(Q+A^\top \tilde{P}_{k+1} A-A^\top  \tilde{P}_{k+1} B(R+B^\top  \tilde{P}_{k+1} B)^{-1}B^\top  \tilde{P}_{k+1} A )}_{\tilde{P}_{k}}x_{k}.
		\end{aligned}$$}
		where $f_2(x,u)=
		(u-K_k x)^\top (R+B^\top F_a(P_{k+1})B)(u-K_k x)$ and $\tilde{P}_{k+1}=F_a({P}_{k+1})$.
		Then, using these identities for $k=\tau-1$,
		$$\begin{aligned}
			&\sum_{k=0}^{\tau-1}z_k^\top z_k-\gamma^2 w_k^\top w_k+x_{\tau}^\top P_\tau x_{\tau}=f_2(x_{\tau-1},u_{\tau-1})+\\
			&\sum_{k=0}^{\tau-2}z_k^\top z_k-\gamma^2 w_k^\top w_k\!+\!x_{\tau-1}^\top P_\tau x_{\tau-1}\!+\!f_1(x_{\tau-1},u_{\tau-1},w_{\tau-1})
		\end{aligned}$$
		Applying the same procedure for $k=\tau-2$, $k=\tau-3$ until $k=0$ we conclude the desired result. \ \ \ \ \ \ \ \ \ \ \ \ \ \ \ \ \ \ \ \ \ \ \ \ 
	\end{proof2}

\begin{lemma}\label{lem:4}
	Consider~\eqref{eq:sys} and $\gamma$ such that $\gamma^2 I-\bar{P}_\gamma > 0$. Then
	\begin{equation}\label{eq:ine}
		\sum_{t=k}^\infty z_t^\top z_t-\gamma^2w_t^\top w_t  \geq x_k^\top G_q x_k,
	\end{equation}
	when
	\begin{equation}\label{eq:wopt}
		w_t = \left\{\begin{aligned}& \bar{L}_{q-(t-k)} (Ax_t+Bu_t), \text{ if }k\leq t< k+q\\
			& 0,  \text{ if }t\geq k+q
		\end{aligned}\right.
	\end{equation} 
	where, for $k \in \{1,\dots,q\}$, 
	\begin{equation}\label{eq:LK} \bar{L}_k = (\gamma^2 I-G_{k-1})^{-1}G_{k-1}
	\end{equation} 
	and, for $k \in \{0,\dots,q-1\}$, 
	\begin{equation}\label{eq:Giter}
		G_{k+1}=F_c(F_a(G_k))
	\end{equation}
	with $G_{0} = P_{\text{LQ}}$ where $P_{\text{LQ}}$ is the unique positive definite solution to the algebraic Riccati equation
	$$ P_{\text{LQ}} = A^\top P_{\text{LQ}} A+P_{\text{LQ}}-A^\top P_{\text{LQ}}B(R+B^\top P_{\text{LQ}}B)^{-1}B^\top P_{\text{LQ}}A. $$
	Moreover, for any $x\in \mathbb{R}^{n_x}$ and $\alpha \in \mathbb{R}_{>0}$, there exists $q \in \mathbb{N}$, denoted by $q=\zeta(x,\alpha)$, such that
	\begin{equation}\label{eq:GPgamma}
		\|x^\intercal G_q x-x^\intercal \bar{P}_\gamma x\|<\alpha.
	\end{equation} 
	Such $q$ can be found by running~\eqref{eq:Giter} until~\eqref{eq:GPgamma} is met.   \hfill $\square$
\end{lemma}

\begin{proof}
	Since~\eqref{eq:sys} is time-invariant it suffices to prove the result for $t=0$, which simplifies the notation. Let 
	$J_{E}(x_r) =\min_{u_t=\mu_{u,t}(\mathcal{J}_t)}\sum_{t=r}^\infty z_t^\top z_t-\gamma^2 w_t^\intercal w_t$
	when $w_t=0$ for every $t\geq r$ in~\eqref{eq:sys}. The standard Linear Quadratic Regulator (LQR) policy is the optimal policy for $u_t$ and leads to the cost $J_{E}(x_r)=x_r^\top P_{LQ} x_r$. 
	Consider now
	$$\begin{aligned}
		&J_{G,r}(x_0)=\\
		&\min_{u_t=\mu_{u,t}(\mathcal{J}_t)}\max_{w_t=\mu_{w,t}(\mathcal{I}_t)}\sum_{t=0}^{r-1} z_t^\top z_t-\gamma^2w_t^\top w_r+J_{E}(x_r)
	\end{aligned}$$
	From standard arguments for quadratic games~\cite[Ch. 3]{basar:91},   $J_{G,r}(x_0) = x_0^\top G_r x_0$, for optimal disturbance policy~\eqref{eq:wopt} and optimal control policy $u_t=K_{q-(t-k)}x_t$, for $t\in \{k,\dots,k+q\}$, $K_k=-(R+B^\top F_a(G_{k-1})B)^{-1}B^\top F_a(G_{k-1})A.$ This implies~\eqref{eq:ine}.  Moreover, $J_{G,r+1}(x_0) \geq J_{G,r}(x_0)$, and $\lim_{r\rightarrow \infty} G_r = \bar{P}_\gamma$, which implies ~\eqref{eq:GPgamma} is met for some $q$, which can be found with the stated method. 
\end{proof}
	
	\bibliography{IEEECL}
	\bibliographystyle{IEEEtran}
\end{document}